\documentclass[a4paper,12pt,reqno]{amsart}
\usepackage{a4wide}
\usepackage{amsmath}
\usepackage{amssymb}
\usepackage{amsthm}
\usepackage{latexsym}
\usepackage{graphicx}
\usepackage[english]{babel}
              
\newtheorem{prop}[subsection]{Proposition}

\newtheorem{teor}[subsection]{Theorem}

\newtheorem{cor} [subsection]{Corollary}
\theoremstyle{definition}

\theoremstyle{remark}

\newtheorem{exm} [subsection]{Example}

\def\hdepth{\operatorname{hdepth}}

\numberwithin{equation}{section}

\begin{document}

\title[On the Hilbert depth of the quotient ring of the edge ideal of a star graph]
      {On the Hilbert depth of the quotient ring of the edge ideal of a star graph}
\author[Silviu B\u al\u anescu, Mircea Cimpoea\c s, Mihai Cipu]{Silviu B\u al\u anescu$^1$, Mircea Cimpoea\c s$^2$ and Mihai Cipu$^3$}  
\date{}



\footnotetext[1]{ \emph{Silviu B\u al\u anescu}, National University of Science and Technology Politehnica Bucharest, Faculty of
Applied Sciences, 
Bucharest, 060042, E-mail: silviu.balanescu@stud.fsa.upb.ro}
\footnotetext[2]{ \emph{Mircea Cimpoea\c s}, National University of Science and Technology Politehnica Bucharest, Faculty of
Applied Sciences, 
Bucharest, 060042, Romania and Simion Stoilow Institute of Mathematics, Research unit 5, P.O.Box 1-764,
Bucharest 014700, Romania, E-mail: mircea.cimpoeas@upb.ro,\;mircea.cimpoeas@imar.ro}
\footnotetext[3]{ \emph{Mihai Cipu}, Simion Stoilow Institute of Mathematics, Research unit nr. 7, P.O.Box 1-764,
Bucharest 014700, Romania, E-mail: mihai.cipu@imar.ro}

\begin{abstract}
Let $S_n=K[x_1,\ldots,x_n,y]$ and $I_n=(x_1y,x_2y,\ldots,x_ny)\subset S_n$ be the edge ideal of star graph.
We prove that $\hdepth(S_n/I_n)\geq \left\lceil \frac{n}{2} \right\rceil + \left\lfloor \sqrt{n} \right\rfloor - 2$.
Also, we show that for any $\varepsilon>0$, there exists some integer $A=A(\varepsilon)\geq 0$ such that
$\hdepth(S_n/I_n)\leq \left\lceil \frac{n}{2} \right\rceil + \left\lfloor \varepsilon n \right\rfloor + A - 2$.
We deduce that $\lim\limits_{n\to\infty} \frac{1}{n}\hdepth(S_n/I_n) = \frac{1}{2}$.
\end{abstract}

\maketitle

\section{Introduction}

Let $K$ be a field and let $S$ be a polynomial ring over $K$.
Let $M$ be a finitely generated graded $S$-module. The Hilbert depth of $M$, denoted by $\hdepth(M)$, is the 
maximal depth of a finitely generated graded $S$-module $N$ with the same Hilbert series as $M$; see
\cite{bruns,uli} for further details.

In \cite[Theorem 2.4]{lucrare2} we proved that if $0\subset I\subsetneq J\subset S$ are two squarefree monomial ideals,
$$\hdepth(J/I)=\max\{q\;:\;\beta_k^q(J/I)=\sum_{j=0}^k (-1)^{k-j}\binom{q-j}{k-j}\alpha_j(J/I)\geq 0\text{ for all }0\leq k\leq q\},$$
where $\alpha_j(J/I)$ is the number of squarefree monomials of degree $j$ in $J\setminus I$, for $0\leq j\leq n$.

Using this new combinatorial characterization of the Hilbert depth, in \cite{lucrare3} we studied this invariant for
several classes of squarefree monomial ideals. In particular, we proved that if $I_n=(x_1y,\ldots,x_ny)\subset S_n:=K[x_1,\ldots,x_n,y]$,
then $\hdepth(I_n)=\left\lfloor \frac{n+3}{2} \right\rfloor$. Note that $I_n$ can be viewed as the edge ideal of a star
graph with $n$ rays. However, the problem of computing $\hdepth(S_n/I_n)$ is very difficult and the aim of our paper is
to tackle this problem and find some sharp bounds for $\hdepth(S_n/I_n)$. 

In Theorem \ref{t24}
we prove that
$\binom{ k - \left\lfloor \sqrt{2k+1} \right\rfloor + 2s+1}{2s} \geq \binom{ k + \left\lfloor \sqrt{2k+1} \right\rfloor - 2}{2s-1},$
for all $k\geq 1$ and $1\leq s\leq  \left\lfloor \sqrt{2k+1} \right\rfloor-1$. Using this combinatorial inequality, we show in
Theorem \ref{t2} that
$$\hdepth(S_n/I_n) \geq \left\lceil \frac{n}{2} \right\rceil + \left\lfloor \sqrt{n} \right\rfloor - 2,\text{ for all }n\geq 1.$$
Using similar techniques, in Theorem \ref{t3} we prove that that for any $\varepsilon>0$, there exists some integer $A=A(\varepsilon)\geq 0$ such that
$$\hdepth(S_n/I_n)\leq \left\lceil \frac{n}{2} \right\rceil + \left\lfloor \varepsilon n \right\rfloor + A - 2.$$
In particular, from Theorem \ref{t2} and Theorem \ref{t3} we deduce in Corollary \ref{cory} that
$$\lim\limits_{n\to\infty} \frac{1}{n}\hdepth(S_n/I_n) = \frac{1}{2}.$$
In the last section, we provide some examples and computer experiments.

\section{Main results}

The star graph $\mathcal S_n$ is the graph on the vertex set $V(\mathcal S_n)=\{x_1,\ldots,x_n,y\}$
with the edge set $E(\mathcal S_n)=\{\{x_i,y\}\;:\;1\leq i\leq n\}$. Let $S_n=K[x_1,\ldots,x_n,y]$. We consider the ideal
$$I_n=(x_1y,x_2y,\ldots,x_ny)\subset S_n,$$
which can be interpreted as the edge ideal of the star graph $\mathcal S_n$.

The following result is a particular case ($m=1$) of \cite[Theorem 2.9]{lucrare3}.

\begin{prop}
We have that $\hdepth(I_n)=\left\lfloor \frac{n+3}{2} \right\rfloor$.
\end{prop}

However, $\hdepth(S_n/I_n)$ seems very difficult to compute, in general.
As a direct consequence of \cite[Theorem 2.6]{lucrare3}, we get:

\begin{prop}
We have that $\hdepth(S_n/I_n)\leq n+2-\left\lceil \sqrt{2n+\frac{1}{4}}+\frac{1}{2} \right\rceil$.
\end{prop}

Also, as a direct consequence of \cite[Theorem 2.7]{lucrare3} and of the identity $\binom{-x}{k}=(-1)^k\binom{x+k-1}{k}$, we get
the following result.

\begin{prop}\label{p3}
We have that 
$$\hdepth(S_n/I_n)=\max\{d\;:\;\binom{n-d+2s-1}{2s}\geq \binom{d-1}{2s-1}\text{ for all }1\leq s\leq \left\lfloor \frac{d}{2} \right\rfloor\}. $$
\end{prop}

\begin{teor}\label{t24}
For all $k\geq 2$ and $1\leq s\leq  \left\lfloor \sqrt{2k+1} \right\rfloor-1$, we have that
$$\binom{ k - \left\lfloor \sqrt{2k+1} \right\rfloor + 2s+1}{2s} \geq \binom{ k + \left\lfloor \sqrt{2k+1} \right\rfloor - 2}{2s-1}.$$
\end{teor}

\begin{proof}
If $s=1$, the conclusion is equivalent to
$$(k - \left\lfloor \sqrt{2k+1} \right\rfloor + 3)(k - \left\lfloor \sqrt{2k+1} \right\rfloor + 2)\geq 2( k + \left\lfloor \sqrt{2k+1} \right\rfloor - 2),$$
which can  easily be proved for all $k\geq 1$. Hence, we may assume $s\geq 2$ and $k\geq 4$. Moreover, using case by case verification,
we can get the conclusion for $1\leq k\leq 400$. So, in the following, we will assume $k\geq 401$ and thus $\sqrt{2k+1}> 28$.

We denote $p=\left\lfloor \sqrt{2k+1} \right\rfloor$ and we consider the following two cases: (i) $p=2r+1$ for some $r\geq 14$ and
(ii) $p=2r$ for some $r\geq 14$.

(i) Since $p=2r+1$ and $p\leq \sqrt{2k+1} < p+1$, it follows that
	$$ 4r^2 + 4r + 1 \leq 2k+1 \leq 4r^2 + 8r + 3.$$
Hence, we can write $k=2r^2+2r+t$ for some $0\leq t\leq 2r+1$. Note that $2\leq s\leq 2r$.
The desired inequality is equivalent to
 \begin{equation}\label{31}
\binom{2r^2+2s+t}{2s} \geq \binom{2r^2+4r+t-1}{2s-1}\text{ for all }r\geq 14,\;0\leq t\leq 2r+1
\text{ and }2\leq s\leq 2r.
\end{equation}

Put $x=2r^2+t$. By \eqref{31}, we have to prove the inequality
$\binom{x+2s}{2s}\geq
\binom{x+4r-1}{2s-1}$, which is equivalent to
$$(x+1)(x+2)\cdots (x+2s)\geq 2s(x+4r-2s+1)(x+4r-2s+2)\cdots (x+4r-1)$$
as well as to
\begin{equation}\label{311}
0\leq \log\left( 1 + \frac{x}{2s} \right) + \sum_{j=1}^{2s-1} \log\left( 1+ \frac{2s-4r}{x+4r-2s+j} \right).
\end{equation}
Note that \eqref{311} is obvious for $s=2r$, so we may assume $s\leq 2r-1$. 

Since $\log(1+y)\geq \frac{y}{1+y}$ for all $y>-1$, in order to prove \eqref{311} it suffices to show
 \begin{equation}\label{3111}
 0 \leq \log\left( 1 + \frac{x}{2s} \right) - (4r-2s)\sum_{j=1}^{2s-1}\frac{1}{x+j}.
 \end{equation}

 We recall the following classical bounds (see \cite{graham}) for the harmonic series:
 \begin{equation}\label{mas}
 \begin{split}
 & \sum_{m\leq X}\frac{1}{m} \leq \log X + \gamma  + \frac{1}{2X} + \frac{1}{4X^2},\\
 & \sum_{m\leq X}\frac{1}{m} \geq \log X + \gamma  - \frac{1}{2X} - \frac{1}{4X^2},
 \end{split}
\end{equation}
 where $\gamma\approx 0.57721$ is the Euler--Mascheroni constant. By applying the above inequalities, we deduce that
 \begin{align*}
 \sum_{j=1}^{2s-1}\frac{1}{x+j} & \leq \log(x+2s-1)+ \frac{1}{2(x+2s-1)}+\frac{1}{4(x+2s-1)^2}-
  \log x + \frac{1}{2x} + \frac{1}{4x^2} \\
& \leq \log\left( 1 + \frac{2s-1}{x} \right) + \frac{1}{2(x+2s-1)}+\frac{1}{300(x+2s-1)} + \frac{1}{2x} +\frac{1}{288x} \\
	& < \log\left( 1 + \frac{2s-1}{x} \right) + \frac{A}{x},
 \end{align*}
with $A=1.00128$.
 Here, we used the fact that $x\geq 2r^2\geq 392$ and $x+2s-1\geq 395$.

 Now, in order to prove \eqref{3111}, using the above upper bound for $\sum_{j=1}^{2s-1}\frac{1}{x+j}$, it is enough to show that
 $$
 \log\left(1+\frac{x}{2s}\right) > (4r-2s)\left( \log\left(1+\frac{2s-1}{x}\right)+\frac{A}{x} \right).
 $$
 As the  expression on the left side is increasing w.r.t. $x$ and the  expression on the right side  is decreasing w.r.t. $x$, it is enough
 to consider the case $x=2r^2$, that is, to prove that
 \begin{equation}\label{wish1}
 \log\left(1+\frac{r^2}{s}\right) > (4r-2s)\left( \log\left(1+\frac{2s-1}{2r^2}\right)+\frac{A}{2r^2}  \right),
 \text{ for all }r\geq 14\text{ and }2\leq s\leq 2r-1.
 \end{equation}

Using the elementary inequality $\log(1+y)<y$ valid for $y>0$ and the arithmetic mean--geometric mean inequality,
the right-hand side is bounded from above as follows:
\[
(4r-2s) \left( \frac{2s-1}{2r^2} + \frac{A}{2r^2}  \right)
< \frac{(2r+0.5A-0.5)^2}{2r^2}= 2+ \frac{A-1}{r} + \frac{(A-1)^2}{8r^2}< 2.002.
\]
The left-hand side is at least
\[
\log \left( 1+ \frac{r^2}{2r-1} \right)\ge \log (1+196/27)>2.111.
\]
Comparison of the last two displayed inequalities ends the proof of part (i).

(ii) Since $p=2r$ and $p\leq \sqrt{2k+1} < p+1$, it follows that
           $$ 4r^2 \leq 2k+1 \leq 4r^2 + 4r.$$
 Hence, we can write $k=2r^2+t$ for some $0\leq t\leq 2r-1$ and the conclusion of the theorem is equivalent to
 \begin{equation}\label{32}
 \binom{2r^2-2r+t+2s+1}{2s}\geq \binom{2r^2+2r+t-2}{2s-1}\text{ for all }r\geq 14,\;0\leq t\leq 2r-1, 
 2\leq s\leq 2r-1.
 \end{equation}

We denote $x=2r^2-2r+t+1$ and we have to prove that $\binom{x+2s}{2s}\geq \binom{x+4r-3}{2s-1}$.
 Similarly to the case (i), this is equivalent to
 \begin{equation}\label{322}
 0\leq \log\left( 1 + \frac{x}{2s} \right) + \sum_{j=1}^{2s-1} \log\left( 1+ \frac{2s+2-4r}{x+4r-2s-2+j} \right).
 \end{equation}
 In order to prove \eqref{322} it suffices to show that
\begin{equation}\label{3222}
 0 \leq \log\left( 1 + \frac{x}{2s} \right) - (4r-2s-2)\sum_{j=1}^{2s-1}\frac{1}{x+j}.
 \end{equation}
 If $s=2r-1$, then \eqref{3222} holds, hence we may assume $s\leq 2r-2$. 
					
 From \eqref{mas} and the fact that $x\geq 2r^2-2r+1\geq 1513$ and $x+2s-1\geq 1516$, it follows that
$$\sum_{j=1}^{2s-1}\frac{1}{x+j}\leq \log\left( 1 + \frac{2s-1}{x} \right) + \frac{A}{x},$$
where this time $A=1.00067$.
 Therefore, in order to prove \eqref{3222} it is enough to show that
 $$
 \log\left(1+\frac{x}{2s}\right) > (4r-2s-2)\left( \log\left(1+\frac{2s-1}{x}\right)+\frac{A}{x} \right).
 $$
 As in  case (i), it is enough to prove the above for $x=2r^2-2r+1$, that is
\begin{equation}\label{wish2}
 \log\left(1+\frac{2r^2-2r+1}{2s}\right) > (4r-2s-2)\left( \log\left(1+\frac{2s-1}{2r^2-2r+1}\right)+
 \frac{A}{2r^2-2r+1}  \right),
\end{equation}
for all $r\geq 14$ and $2\leq s\leq 2r-2$.

Since $\log(1+y)<y$ for $y>0$,
 in order to prove \eqref{wish2} it suffices to show that
 \begin{equation*}\label{wis_2}
 \log\left(1+\frac{2r^2-2r+1}{2s}\right) >  \frac{(4r-2s-2)(2s+A-1)}{2r^2-2r+1},
 \text{ for all }r\geq 14, \ 2\leq s\leq 2r-2.
 \end{equation*}
This is true as seen from the chain of inequalities
\begin{align*}
\log\left(1+\frac{2r^2-2r+1}{2s}\right) & \ge \log\left(1+\frac{2r^2-2r+1}{4r-4}\right) \ge \log\left(1+\frac{365}{52}\right) > 2.081 \\
& > 1.998 >\frac{(2r+0.5A-1.5)^2}{2r^2-2r+1}
\end{align*}
obtained by arguments similar to those employed in the previous case.
  Hence, the proof is complete.
\end{proof}

\begin{teor}\label{t2}
For all $n\geq 1$, we have that $\hdepth(S_n/I_n)\geq \left\lceil \frac{n}{2} \right\rceil + \left\lfloor \sqrt{n} \right\rfloor - 2$.
\end{teor}

\begin{proof}
The conclusion is trivially true for $n\le 2$, so for the rest of the proof we assume $n\ge 3$.
Let $d=\left\lceil \frac{n}{2} \right\rceil + \left\lfloor \sqrt{n} \right\rfloor - 2$. 

If $n=2k+1$, then $d=k+\left\lfloor \sqrt{2k+1} \right\rfloor -1$ and
\begin{equation}\label{ecc1}
\binom{n-d+2s-1}{2s} = \binom{ k - \left\lfloor \sqrt{2k+1} \right\rfloor + 2s+1}{2s} \text{ and } \binom{d-1}{2s-1} = 
                       \binom{ k + \left\lfloor \sqrt{2k+1} \right\rfloor - 2}{2s-1} .
\end{equation}
Also, if $n=2k$,  $d=k+\left\lfloor \sqrt{2k} \right\rfloor - 2$ and
\begin{equation}\label{ecc2}
\binom{n-d+2s-1}{2s} = \binom{ k - \left\lfloor \sqrt{2k} \right\rfloor + 2s+1}{2s} \text{ and } \binom{d-1}{2s-1} = 
                       \binom{ k + \left\lfloor \sqrt{2k} \right\rfloor - 3}{2s-1} .
\end{equation}
Since $\left\lfloor \sqrt{2k+1} \right\rfloor\geq  \left\lfloor \sqrt{2k} \right\rfloor$, in order to prove the
inequality 
\begin{equation}\label{vrem}
\binom{n-d+2s-1}{2s}\geq \binom{d-1}{2s-1}\text{ for }d=\left\lceil \frac{n}{2} \right\rceil + \left\lfloor \sqrt{n} \right\rfloor - 2\text{ and }1\leq s\leq \left\lfloor \frac{d}{2} \right\rfloor,
\end{equation}
it is enough to consider the case $n=2k+1$ and to show that
\begin{equation}\label{ecu}
\binom{ k - \left\lfloor \sqrt{2k+1} \right\rfloor + 2s+1}{2s} \geq \binom{ k + \left\lfloor \sqrt{2k+1} \right\rfloor - 2}{2s-1}
\text{ for all }1\leq s\leq \left\lfloor \frac{k + \left\lfloor \sqrt{2k+1} \right\rfloor -2 }{2} \right\rfloor.
\end{equation}
Obviously, \eqref{ecu} holds for $s\geq  \left\lfloor \sqrt{2k+1} \right\rfloor$. From Theorem \ref{t24} and the above discussions, it follows that equation \eqref{vrem} holds. Hence, the conclusion follows from Proposition \ref{p3}.
\end{proof}

\begin{teor}\label{t3}
For any $\frac{1}{2}>\varepsilon>0$
define the positive integer
$A=A(\varepsilon)$ by
\[
A=   \left\lceil  \frac{3B+6}{8\varepsilon}-\frac{B+1}{2} \right\rceil , \]
where $B=\log\bigl( \frac{2\varepsilon +1}{4\varepsilon}\bigr)$.
Then
$$\hdepth(S_n/I_n)\leq \left\lceil \frac{n}{2} \right\rceil + \left\lfloor \varepsilon n \right\rfloor + A - 2,
\text{ for all }n\geq 2.$$
\end{teor}

\begin{proof}
Let $\varepsilon>0$. Having in view Proposition \ref{p3}, the conclusion of the theorem is equivalent to the fact that there exists 
some nonnegative integer $A$ such that, for any $n\geq 2$ and $d=\left\lceil \frac{n}{2} \right\rceil + \left\lfloor \varepsilon n \right\rfloor + A - 1$,
there exists some $1\leq s\leq \left\lfloor \frac{d}{2} \right\rfloor$ for which
$$\binom{n-d+2s-1}{2s} < \binom{d-1}{2s-1}$$
or equivalently
\begin{equation}\label{condy}
\binom{\left\lfloor \frac{n}{2} \right\rfloor - \left\lfloor \varepsilon n \right\rfloor - A + 2s}{2s} < 
  \binom{\left\lceil \frac{n}{2} \right\rceil + \left\lfloor \varepsilon n \right\rfloor + A - 2}{2s-1}.
\end{equation}

If $n=2k$, then $d=k+\left\lfloor 2k\varepsilon \right\rfloor + A - 1$ and \eqref{condy} becomes
\begin{equation}\label{con1}
  \binom{k - \left\lfloor 2k\varepsilon  \right\rfloor - A + 2s}{2s} < 
  \binom{k + \left\lfloor 2k\varepsilon  \right\rfloor + A - 2}{2s-1},
\end{equation}
$\text{ for some }1\leq s \leq \left\lfloor \frac{k+\left\lfloor 2k\varepsilon \right\rfloor + A - 1}{2} \right\rfloor$.

If $n=2k+1$, then $d=k+\left\lfloor (2k+1)\varepsilon \right\rfloor +  A$ and \eqref{condy} is equivalent to
\begin{equation}\label{con2}
  \binom{k - \left\lfloor (2k+1)\varepsilon \right\rfloor - A + 2s}{2s} < 
  \binom{k + \left\lfloor (2k+1)\varepsilon \right\rfloor + A - 1}{2s-1},
\end{equation}
$\text{ for some }1\leq s\leq \left\lfloor \frac{k+\left\lfloor (2k+1)\varepsilon \right\rfloor + A}{2} \right\rfloor$.

Since $\binom{k + \left\lfloor (2k+1)\varepsilon \right\rfloor + A - 1}{2s-1} > \binom{k + \left\lfloor 2k\varepsilon  \right\rfloor + A - 2}{2s-1}$
and $\binom{k - \left\lfloor (2k+1)\varepsilon \right\rfloor - A + 2s}{2s}\le \binom{k - \left\lfloor 2k\varepsilon  \right\rfloor  - A + 2s}{2s}$, while $\left\lfloor \frac{k+\left\lfloor 2k\varepsilon \right\rfloor + A - 1}{2} \right\rfloor \leq \left\lfloor \frac{k+\left\lfloor (2k+1)\varepsilon \right\rfloor + A}{2} \right\rfloor$,  it
is enough to consider only the case $n=2k$.

Let $A\geq 0$ and $x=k - \left\lfloor 2k\varepsilon  \right\rfloor - A$. Without any loss of generality,            we may assume $x\geq 0$, otherwise \eqref{con1} is trivially true. Note that
\eqref{con1}
is equivalent to  $$ (x+1)(x+2)\cdots (x+2s) < 2s(x+2 \left\lfloor 2k\varepsilon  \right\rfloor + 2A -2s) \cdots
(x+2 \left\lfloor 2k\varepsilon  \right\rfloor+2A-2),$$
as well as to $$\log\left( 1+\frac{x}{2s} \right) < \sum_{j=1}^{2s-1} \log\left(1+ \frac{2 \left\lfloor 2k\varepsilon\right\rfloor + 2A - 2s -1}{x+j}\right).$$

We claim that $s=\left\lfloor 2k\varepsilon\right\rfloor$ does the job. Explicitly, this means
there exists $A\geq 0$ which depends only on $\varepsilon$ such that, for all $k\geq 1$ with
$k - \left\lfloor 2k\varepsilon  \right\rfloor\geq A$, we have
\begin{equation}\label{coni1}
\log\left( 1+\frac{k - \left\lfloor 2k\varepsilon  \right\rfloor - A}{2\left\lfloor 2k\varepsilon  \right\rfloor} \right) <
\sum_{j=1}^{2\left\lfloor 2k\varepsilon  \right\rfloor -1} \log\left(1+ \frac{2A -1}{k - \left\lfloor 2k\varepsilon  \right\rfloor - A+j}\right).
\end{equation}
Since for $A\ge 1/(2\varepsilon)$ it holds
$$1+\frac{k - \left\lfloor 2k\varepsilon  \right\rfloor - A}{2\left\lfloor 2k\varepsilon  \right\rfloor} = 1 +
\frac{k-A}{2\left\lfloor 2k\varepsilon  \right\rfloor} - \frac{1}{2} < \frac{1}{2} + \frac{1}{4\varepsilon} =
\frac{2\varepsilon+1}{4\varepsilon},$$
in order to prove \eqref{coni1} it is enough to find $A\geq 1/(2\varepsilon)$ such that, for all $k\geq 1$ with
$k - \left\lfloor 2k\varepsilon  \right\rfloor\geq A$, we have
\begin{equation}\label{conn1}
\log\left( \frac{2\varepsilon+1}{4\varepsilon} \right) < \sum_{j=1}^{2\left\lfloor 2k\varepsilon  \right\rfloor -1} \log\left(1+ \frac{2A -1}{k - \left\lfloor 2k\varepsilon  \right\rfloor+j}\right).
\end{equation}
Let $B:=\log\left( \frac{2\varepsilon+1}{4\varepsilon} \right)$ .
Since $\log(1+y)>\frac{y}{y+1}$ for $y>0$, in order to prove \eqref{conn1} it is enough to find $A\geq 1/(2\varepsilon)$ such that, for all $k\geq 1$ with
$k - \left\lfloor 2k\varepsilon  \right\rfloor\geq A$, we have
\begin{equation}\label{connn1}
\sum_{j=1}^{2\left\lfloor 2k\varepsilon  \right\rfloor -1} \frac{2A-1}{k-\left\lfloor 2k\varepsilon  \right\rfloor+j+(2A-1)} > B.
\end{equation}
Observe that for the left-hand side  we have
\begin{align*}
\sum_{j=1}^{2\left\lfloor 2k\varepsilon  \right\rfloor -1} \frac{2A-1}{k-\left\lfloor 2k\varepsilon  \right\rfloor+j+(2A-1)}
&= (2A-1)\sum_{\ell=1-\left\lfloor 2k\varepsilon  \right\rfloor}^{\left\lfloor 2k\varepsilon  \right\rfloor-1}\frac{1}{2A+k-1-\ell}  \\
& > \frac{(2A-1)(2\left\lfloor 2k\varepsilon  \right\rfloor -1)}{2A+k-1} > \frac{(2A-1)(4\varepsilon k-3)}{2A+k-1}.
\end{align*}
Hence, it suffices to find $A\geq 1/(2\varepsilon)$ such that, for all $k\geq 1$ with
$k - \left\lfloor 2k\varepsilon  \right\rfloor\geq A$, we have
$$(2A-1)(4\varepsilon k-3) > (2A+k-1)B,$$
which is equivalent to
\begin{equation} \label{eq:nou}
(2A-1)\left(4\varepsilon - \frac{B+3}{k}\right)>B.
\end{equation}
As $k\ge \lfloor 2k\varepsilon \rfloor +A>A+2k\varepsilon -1$, we have
\[
 (2A-1)\left(4\varepsilon - \frac{B+3}{k}\right) > (2A-1)\left(4\varepsilon - \frac{(B+3)(1-2\varepsilon)}{A-1}\right).
\]
Hence, in order to complete the proof it suffices to find $A\geq 1/(2\varepsilon)$ such that, for all $k\geq 1$ with
$k - \left\lfloor 2k\varepsilon  \right\rfloor\geq A$, we have
\begin{equation}\label{eq:nout}
(2A-1)\left(4\varepsilon - \frac{(B+3)(1-2\varepsilon)}{A-1}\right)>B.
\end{equation}
We denote $y:=2A-1$ and $N:=(B+3)(1-2\varepsilon)$. We have $A-1=\frac{y-1}{2}$ and thus \eqref{eq:nout} is equivalent to
$$4\varepsilon y - \frac{2yN}{y-1}>B,$$
which is also equivalent to
\begin{equation}\label{fdx}
f(y):=4\varepsilon y^2 - (4\varepsilon + 2N + B)y + B > 0.
\end{equation}
Since the discriminant of $f(y)$ is $$(4\varepsilon + 2N + B)^2-16\varepsilon B = (4\varepsilon - B)^2 + 4N^2 + 16\varepsilon N + 4NB>0,$$
it follows that $f(y)>0$ for
$$y > \frac{4\varepsilon + 2N + B + \sqrt{(4\varepsilon + 2N + B)^2-16\varepsilon B}}{8\varepsilon}.$$
Henceforth, we have $f(y)>0$ for
$$y > \frac{4\varepsilon + 2N + B}{4\varepsilon}.$$
Since $A=\frac{y+1}{2}$, it follows that $f(y)>0$ for
$$A>\frac{8\varepsilon + 2N + B}{8\varepsilon} = \frac{3B+6-4B\varepsilon-4\varepsilon}{8\varepsilon}=\frac{3B+6}{8\varepsilon}-\frac{B+1}{2}.$$

Since this lower bound on $A$ is greater than $1/(2\varepsilon)$, the proof is complete.
\end{proof}




\begin{cor}\label{cory}
We have that $$\lim_{n\to\infty} \frac{\hdepth(S_n/I_n)}{n} = \frac{1}{2}.$$
\end{cor}

\begin{proof}
From Theorem \ref{t2} it follows that 
\begin{equation}\label{eks1}
\frac{\hdepth(S_n/I_n)}{n} \geq \frac{1}{2} + \frac{\sqrt{n}-3}{n}\text{ for all }n\geq 1.
\end{equation}


On the other hand, according to Theorem \ref{t3} it follows that for any $\varepsilon>0$ there
exists some $A(\varepsilon)\geq 0$ such that 
\begin{equation}\label{eks2}
\frac{\hdepth(S_n/I_n)}{n} \leq \frac{1}{2} + \varepsilon + \frac{A(\varepsilon)-1}{n}\text{ for all }n\geq 1.
\end{equation}
Now, let $\delta$ be positive and less than $1$,  and put $\varepsilon:=\frac{\delta}{2}$. Then, there exists $n_{\delta}$ such that for all $n\geq n_{\delta}$
we have $\frac{\sqrt{n}-3}{n}<\delta$ and $\frac{A(\varepsilon)-1}{n}<\varepsilon$. From \eqref{eks1} and \eqref{eks2}
it follows that $$\left|\frac{\hdepth(S_n/I_n)}{n}-\frac{1}{2}\right|<\delta,\text{ for all }n\geq n_{\delta}.$$
Thus, we get the required conclusion.
\end{proof}
										
\section{Computer experiments}

For all $n\geq 1$ we consider:
$$\alpha(n)=\hdepth(S_n/I_n)\text{ and }\beta(n)= \left\lceil \frac{n}{2} \right\rceil + \left\lfloor \sqrt{n} \right\rfloor - 2.$$
According to our computer experiments, for $1\leq n\leq 400$, we have:
\begin{align*}
& \alpha(n)-\beta(n)=0,1,1,0,0,1,0,1,0,0,0,1,0,1,0,0,0,0,0,0,0,1,0,1,0,0,0,0,0,0,0,1,0,1,0,\\
&  0,0,0,0,0,0,1,0,1,0,1,0,1,0,0,0,0,0,0,0,1,0,1,0,1,0,1,1,0,0,0,0,0,0,0,0,1,0,1,0,1,0,\\
&  1,1,1,0,0,0,0,0,0,0,1,0,1,0,1,0,1,1,1,1,1,1,0,0,0,0,1,0,1,0,1,0,1,0,1,0,1,1,1,1,1,1,\\
&  1,0,0,0,1,0,1,0,1,0,1,0,1,1,1,1,1,1,1,1,1,1,1,1,1,0,1,0,1,0,1,0,1,1,1,1,1,1,1,1,1,1,\\
&  1,1,2,1,2,1,2,0,1,0,1,0,1,0,1,1,1,1,1,1,1,1,1,1,1,1,2,1,2,1,2,1,2,1,1,0,1,1,1,1,1,1,\\
&  1,1,1,1,1,1,1,1,2,1,2,1,2,1,2,1,2,1,2,2,2,1,1,1,1,1,1,1,1,1,1,1,2,1,2,1,2,1,2,1,2,1,\\
&  2,1,2,2,2,2,2,2,2,2,1,1,1,1,1,1,2,1,2,1,2,1,2,1,2,1,2,1,2,1,2,2,2,2,2,2,2,2,2,2,2,2,\\
&  2,1,2,1,2,1,2,1,2,1,2,1,2,1,2,2,2,2,2,2,2,2,2,2,2,2,2,2,2,2,3,2,3,2,3,2,2,1,2,1,2,1,\\
&  2,1,2,2,2,2,2,2,2,2,2,2,2,2,2,2,2,2,3,2,3,2,3,2,3,2,3,2,3,2,3,1,2,2,2,2,2,2,2,2,2,2,\\
&  2,2,2,2,2,2,3,2,3,2,3,2,3,2,3,2,3,2,3,2,3,2,3,3,3,3,3,3,2.
\end{align*}
This shows that the lower bound given Theorem \ref{t2} is quite sharp.

\begin{exm}\rm\label{ex16}
Let $\varepsilon=\frac{1}{6}$. With formula given in  Theorem \ref{t3}, we get 
$$A=  \left \lceil \frac{7\log 2 + 16}{4} \right\rceil  = 6,$$ so
 we have
$$\hdepth(S_n/I_n)\leq \left\lceil \frac{n}{2} \right\rceil + \left\lfloor \frac{n}{6} \right\rfloor + 4, \text{ for all }n\geq 2.
$$
Note that, if we consider Eq.~\eqref{eq:nou}, we need to find $A\geq 3$
such that
$$(2A-1)\left(1 - \frac{3+\log 2}{A-1}\right) > \frac{3}{2}\log(2).
$$
The smallest value of $A$ which satisfies this inequality is indeed $6$. 

However, this value of the constant $A$ is not optimal. For instance,
according to our computer experiments, we have
$$\hdepth(S_n/I_n)\leq \left\lceil \frac{n}{2} \right\rceil + \left\lfloor \frac{n}{6} \right\rfloor, \text{ for all }n\geq 2.$$
So, it seems, the smallest possible value for $A$ is $2$.
\end{exm}

\end{document}